\documentclass[11pt]{article}
\usepackage{amssymb,amsmath,latexsym,amsthm,array}
\usepackage{booktabs}
\usepackage{url}
\usepackage{multirow}
\usepackage{float}

\usepackage[linesnumbered, lined,boxed, commentsnumbered]{algorithm2e}
\usepackage{xcolor}

\newenvironment{frcseries}{\fontfamily{frc}\selectfont}{}

\newtheorem{thm}{Theorem}
\newtheorem{lemma}[thm]{Lemma}
\newtheorem{cor}[thm]{Corollary}

\theoremstyle{definition}
\newtheorem{defn}[thm]{Definition}
\newtheorem{rem}[thm]{Remark}

\newtheorem{conj}{Conjecture}

\makeatletter

\renewcommand\section{\@startsection {section}{1}{\z@}%
{-30pt \@plus -1ex \@minus -.2ex}%
{2.3ex \@plus.2ex}%
{\normalfont\normalsize\bfseries}}

\renewcommand\subsection{\@startsection{subsection}{2}{\z@}%
{-3.25ex\@plus -1ex \@minus -.2ex}%
{1.5ex \@plus .2ex}%
{\normalfont\normalsize\bfseries}}
\renewcommand{\@seccntformat}[1]{\csname the#1\endcsname. } 

\makeatother

\begin{document}

\begin{center}
\uppercase{\bf   On Hales-Jewett  and Related Numbers}
\vskip 20pt
{\bf Nathan Conlon}\footnote{The first author's contribution was done as part of an undergraduate
thesis in 2021 while an undergraduate at Colgate University
under the advisorship of the third author.}\\
{\it University of California, San Diego, San Diego, California}\\ 
{\tt nconlon@ucsd.edu}\\
\vskip 10pt
{\bf Nayda Farnsworth}\footnote{The second author's contribution was done as part of an undergraduate
thesis and independent study in 2025-26 at Colgate University
under the advisorship of the third author.}\\
{\it Colgate University, Hamilton, New York}\\ 
{\tt nfarnsworth@colgate.edu}\\
\vskip 10pt
{\bf Aaron Robertson} \\
{\it Department of Mathematics,}
{\it Colgate University, Hamilton, New York}\\ 
{\tt arobertson@colgate.edu}\\
\end{center}

\begin{abstract}
\noindent
We investigate the Hales-Jewett numbers and some variants of them,
both computationally and via enumerative and probabilistic arguments.  
In particular, we give the improved
lower bound formula on the combinatorial-geometric
variant of the Hales-Jewett numbers.
{As a consequence of one of the variants investigated, we show
that the Milton Bradley game Connect Four played with 2 players
on a 5-dimensional hypercube of any side length can end in a draw.
}
\end{abstract}

\section{Introduction}

The Hales-Jewett Theorem is a cornerstone in Ramsey Theory, asserting that for any  $k, r \in \mathbb{Z}^+$, there exists a dimension $n$ such that for all $d \geq n$, every $r$-coloring of the discrete hypercube $[k]^d$ admits a monochromatic combinatorial line. Analogously, this result may be viewed as a high-dimensional generalization of Tic-Tac-Toe: in sufficiently many dimensions, a draw is impossible. 

 The least such dimension $n$ is called the \textit{Hales-Jewett number}, denoted $HJ(k;r)$. These numbers are notoriously difficult to determine, with the only known nontrivial value being $HJ(3;2) = 4$, established by Hindman and Tressler \cite{HT}. Even for small dimensions, the search space grows astronomically. For example, in the case of 2-colorings of $[4]^4$, there are $2^{256} \approx 1.16 \times 10^{77}$ possible colorings, a quantity comparable to the number of atoms in the observable universe. 

This combinatorial explosion, together with a lack of general structural techniques, has left most Hales-Jewett numbers' values poorly understood. Aside from a handful of small cases, both exact values and sharp bounds remain largely unknown. 

In this work, we approach these problems by investigating related numbers and using efficient computational techniques. We encode Hales-Jewett instances as Boolean formulas and leverage  SAT solvers to systematically search for colorings that avoid monochromatic lines. 

We introduce a method for extending combinatorial constructions reliant on van der Waerden numbers via SAT solvers, which yields the new bound $HJ(3;3) \geq 14$.

In addition to our computational results, we introduce  off-diagonal Hales-Jewett numbers and  investigate the combinatorial-geometric variant. We establish new exact values and improved bounds for these numbers using combinatorial, probabilistic, and computational methods.

{
One of the variants, the combinatorial-geometric
variant with partial lines, allows us to investigate the Milton Bradley game Connect Four.  In particular, we
show that, unlike Tic-Tac-Toe where we are
guaranteed a winner on a $3 \times 3 \times 3$ board, Connect Four may end in a tie on
any sized board in up to 5 dimensions.
}


\subsection{Definitions and Notations}

We first introduce notation relevant to discussing the Hales-Jewett Theorem formally.

For $k \in \mathbb{Z}^+$, let $[k] = \{1,2,\dots,k\}$ and write $[k]^n$ for the set of points $p=(p_1,\dots,p_n)$ with $p_i \in [k]$. We view $[k]^n$ as the  discrete $n$-dimensional hypercube of side length $k$.

\begin{defn}[Variable Word, Combinatorial Line]
    A \emph{variable word} is a vector $w(x) \in ([k]\cup\{x\})^n$ containing at least one occurrence of $x$. Let $S(w)=\{i \in [n] : w_i=x\}$ denote the set of variable positions. For $a \in [k]$, define $w(a)\in [k]^n$ by replacing each occurrence of $x$ in $w(x)$ with $a$.
The \emph{combinatorial line} determined by $w(x)$ is
$
L(w)=\{w(a) : a \in [k]\}.
$
\end{defn}

The formal definition of the Hales-Jewett Theorem and Hales-Jewett Number $HJ(k;r)$ follow.

\begin{thm}[Hales-Jewett Theorem]\label{thm2}
For every $k,r \in \mathbb{Z}^+$, there exists a positive integer $HJ(k;r)$ such that for all $n \ge HJ(k;r)$, every $r$-coloring of the points in the discrete hypercube $[k]^n$ admits a monochromatic combinatorial line.
\end{thm}

The variants of the Hales-Jewett numbers that we investigate in this paper are defined in Table~\ref{tab:HJnotation}, below.

\begin{table}[ht!]\small
\centering
\renewcommand{\arraystretch}{1.2}
\begin{tabular}{r| l | l}
\toprule
\textbf{Notation} & \textbf{Description} &\textbf{Paper Location} \\
\midrule
$HJ(k;r)$ & Classical $r$-color Hales-Jewett\\& number for combinatorial lines &Theorem \ref{thm2}\\

$HJ(k,\ell;2)$ & Off-diagonal 2-color Hales-Jewett\\& number (combinatorial lines)&Definition \ref{def11} \\

$HJ^*(k;r)$ & $r$-color CG-Hales-Jewett number (for\\& combinatorial-geometric (cg) lines)&Definition \ref{def20} \\

$HJ^*(k,\ell;2)$ & Off-diagonal 2-color CG-Hales-Jewett\\& number &Definition \ref{def20}\\

$G(k;r)$ & $r$-color Hales-Jewett number for\\& geometric lines&Theorem \ref{thm21} \\

$G(k,\ell;2)$ & Off-diagonal 2-color Hales-Jewett number\\& for geometric lines&Theorem \ref{thm21} \\
$CF(k;r)$&Connect Four number&Definition \ref{def24}\\
\bottomrule
\end{tabular}
\caption{Notations for Hales-Jewett Variants}
\label{tab:HJnotation}
\end{table}

\subsection{Known Results}

We start by summarizing the known  values of the classical Hales-Jewett numbers. 

It is straightforward to verify that $HJ(2;r) = r$ for all $r \in \mathbb{Z}^+$, a consequence of a simple combinatorial argument. The first nontrivial case was resolved by Hindman and Tressler \cite{HT}, who showed that $HJ(3;2) = 4.$
Their proof combines combinatorial argument with computer-assisted verification.

Bounds for the next smallest Hales-Jewett numbers are extremely broad. We provide them below. 
\[
12 \le HJ(4;2) \le 10^{11}; \qquad 13 \le HJ(3;3) \le 3^{36}+732 \approx 10^{17.2}.
\]
 The upper bounds on $HJ(4;2)$ and $HJ(3;3)$ are due to work by Lavrov \cite{Lavrov} and D. Conlon \cite{Conlon}, respectively. The lower bounds on these Hales-Jewett numbers are due to a lemma on the van der Waerden numbers (see Corollary \ref{vdwmapcor}, below), defined next. 

 \begin{rem} Very recently a preprint \cite{MH} gave the improved lower bounds $HJ(3;3) \geq 22$ and $HJ(4;2) \geq 14$.
\end{rem}

\begin{defn}[Van der Waerden Number]
For integers $r\in \mathbb{Z}^+$ and $k_1,k_2,\dots,k_r$, the \emph{van der Waerden number} $w(k_1,k_2,\dots,k_r;r)$ is the smallest
integer $N$ such that every $r$-coloring of $\{1,2,\dots,N\}$ admits a monochromatic $k_i$-term arithmetic progression of color $i$ for
some $i \in \{1,2,\dots,r\}$.  In the situation where
$k=k_1=k_2=\cdots=k_r$ we use the notation $w(k;r)$.
\end{defn}

The lemma by which we obtain van der Waerden-based lower bounds on Hales-Jewett numbers follows.

\begin{lemma}\label{genlem}
Let $f(x): \mathbb{Z}^{\geq 0} \rightarrow \mathbb{Z}^{\geq 0}$.
Let $k,r,d\in\mathbb{Z}^+$ and define
$$N = 1+d \cdot\max_{0 \leq i \leq k-1} f(i).$$ If there exists an $r$-coloring of
$[1,N]$ that avoids  monochromatic sequences of the form
$a+bf(i), i=0,1,\dots,k-1$, then $HJ(k;r)>d$.
\end{lemma}

\begin{proof}
Use the mapping $(a_1,a_2,\dots,a_d) \in [0,k-1]^d$ to $1+\sum_{i=1}^d f(a_i) \in [1,N]$ and color the points in the hypercube via the color under this mapping.
Note that a variable word $w(x)$ maps to $a+bf(x)$ for some positive integers $a$ and $b$.
Evaluating this word at $x=0,1,\dots,k-1$, we get
a sequence of the form $a+bf(i), i=0,1,\dots,k-1$.
Since the coloring of $[1,N]$ avoids monochromatic such progressions,
every combinatorial line is non-monochromatic. 
\end{proof}

We often use $f(x)=x$ in Lemma \ref{genlem} and, as such,
make the following definition.

\begin{defn}[Van der Waerden Mapping]\label{def:vdwmap}
    We will refer to the mapping from $[0,k-1]^d$ to $[1,1+d(k-1)]$
    given by
    \[(a_1, \dots, a_d) \mapsto 1 + \sum_{i=1}^{d} a_i\]
    as a {\it van
der Waerden mapping}.
\end{defn}

Under the van der Waerden mapping 
with $N=w(k;r)-1$ in Lemma \ref{genlem}  we obtain the following known result.

\begin{cor}\label{vdwmapcor}
We have $HJ(k;r) \geq \left\lfloor\frac{w(k;r)-2}{k-1}\right\rfloor+1.$
\end{cor}


\section{Improved Lower Bound}

Our  contribution in the classical Hales-Jewett area is an improved lower bound on $HJ(3;3)$ (see Remark 9, below). In this section, we present the improved bound.  Our approach is to use the van der Waerden mapping to color one slice
of the $[3]^{13}$ hypercube and use a SAT solver to assign colors to
the remaining points of the hypercube.

\begin{thm}\label{thm7} We have $HJ(3;3) \geq 14$.

\end{thm}

\begin{proof}
We obtain a monochromatic line-avoidant coloring of $[3]^{12}$ via the van der Waerden mapping described in Corollary~\ref{vdwmapcor}. We refer to this coloring as $C_{12}$.

We then define a partial coloring of $[3]^{13}$ by
\[
C_{13}(p_1,\dots,p_{12}, 1) = C_{12}(p_1,\dots,p_{12}),
\]
which colors one 12-dimensional slice of the 13-dimensional hypercube.

We encode the partial coloring of $[3]^{13}$ as a Boolean formula under a partial assignment (see Section 6) and use a SAT solver to determine whether it can be extended fully. The resulting formula is satisfiable and yields nine monochromatic line-avoidant colorings, completing the proof. The compute time for this instance was 8831 CPU hours, equivalent to 367 days. We completed the computation using CryptoMiniSat and distributed the instance across 15 cores. 
\end{proof}

\begin{rem}
     To the best of our knowledge, this is the first
instance of a lower bound on the Hales-Jewett numbers that is larger
than that given by the van der Waerden mapping and Corollary \ref{vdwmapcor}.
However, while writing this paper, the authors
discovered that Mouhib and Halbeisen \cite{MH}
proved that $HJ(3;3) \geq 22.$ In their paper,
they acknowledge Theorem \ref{thm7}, which was
presented by the second author at the Experimental Math Seminar at Rutgers.  We have
verified their improved bound.
\end{rem}

The lower bound $HJ(3;3) \geq 14$ serves as a primary demonstration of the effectiveness of this approach.

\section{Off-Diagonal Hales-Jewett Numbers}

Rather than requiring all colors to avoid monochromatic combinatorial lines, we consider colorings in which one color avoids full lines while others avoid partial lines. This formulation may be viewed as a higher-dimensional analogue of the off-diagonal van der Waerden number $w(k_1,k_2,\dots,k_r;r)$.

From a computational perspective, such off-diagonal variants often admit more tractable encodings, making them a natural setting in which to develop and test algorithmic approaches to Hales-Jewett-type problems. We begin by defining partial combinatorial lines.

\begin{defn}[Partial Combinatorial Line]
Let $n,k \in \mathbb{Z}^+$, and let $L(w)=\{w(a):a\in[k]\}$ be a combinatorial line in $[k]^n$.  
A \emph{partial combinatorial line of length $m$} is any subset of the form
$
\{\,w(a): a \in \{i+1,\dots,i+m\}\,\},
$
where $0 \le i \le k-m$. We refer to this subset as a \textit{(combinatorial) $m$-line.}
\end{defn}

We now define the off-diagonal Hales-Jewett numbers.

\begin{defn}[Off-Diagonal Hales-Jewett Number]\label{def11}
For $r \in \mathbb{Z}^+$, let $k_1,k_2,  \dots, k_r \in \mathbb{Z}^+$
and define $k = \max_i k_i$.
The \emph{$r$-color off-diagonal Hales-Jewett number}\break$HJ(k_1,k_2,\dots,k_r;r)$ is the least integer $n$ such that every $r$-coloring of $[k]^n$ contains  
a monochromatic $k_i$-line of color $i$ for some $i \in \{1,2,\dots,r\}$.
When $r = 2$, we may write $HJ(k,\ell)$ in place of $HJ(k,\ell;2)$ for brevity.
\end{defn}

We now present results on certain off-diagonal Hales-Jewett numbers. Computational results for
off-diagonal Hales-Jewett numbers can be found in
Section \ref{sec:comp}.

\begin{thm}\label{thm1} For $k \in \mathbb{Z}^+$, we have $HJ(k,2) \leq k$.
\end{thm}

\begin{proof} Suppose, for a contradiction, that $\chi$ is a 2-coloring of $[k]^k$ with no
red $k$-line and no blue $2$-line.  Then, for any $a_1,a_2,\dots,a_{k-1} \in [k]$ there exists
$y \in [k]$ such that $\chi(a_1,a_2,\dots,a_{k-1},y)$ is blue; otherwise, we would have a red $k$-line.

Let $a^{j}$ denote a string of $j$ consecutive $a$'s.
Given $a \in [k-1]$, 
choose $y_1,  \dots, y_k \in [k]$ so that
$$
\chi(a^{k-i} ,(a+1)^{i-1},y_i) = \mbox{blue} \quad \mbox{for $i \in [k]$}.
$$
To avoid a blue $2$-line, we must have $y_i \neq y_j$ for $i \neq j$, else (assuming $j>i$)
for the variable word  $w(x) = (a^{k-j}, x^{j-i},(a+1)^{i-1},y_i)$ 
we have that both $w(a)$ and $w(a+1)$ are blue.

Next, we know that for some $c \in [k]$, the point $(c,c,\dots,c)$ is blue.

\noindent
{\tt Case 1.} $c=k$.  Let $x_1,x_2,\dots,x_k$ be distinct integers from $[k]$ such that
$$
\chi((c-1)^{k-i} ,c^{i-1},x_i) = \mbox{blue}, \quad \mbox{$1 \leq i \leq k$}.
$$
By choice of $c$ we may take $x_k=c$.  Since the $x_i$ are distinct,
let $x_j=c-1$ for some $j \in [1,k-1]$.  By considering the variable word $w(x) = (x^{k-j},c^{j-1},x)$ we
see that both $w(c-1)$ and $w(c)$ are blue, thereby creating a blue 2-line, a contradiction.

\noindent
{\tt Case 2.} $c<k$. Let $x_1,x_2,\dots,x_k$ be distinct integers from $[k]$ such that
$$
\chi(c^{k-i} ,(c+1)^{i-1},x_i) = \mbox{blue}, \quad \mbox{$1 \leq i \leq k$}.
$$ 
By choice of $c$ we may take $x_1=c$ and let $x_j=c+1$ for
some $j \in [2,k]$. By considering the
variable word $w(x) = (c^{k-j},x^j)$ we see that both
$w(c)$ and $w(c+1)$ are blue, a contradiction.

Since no valid coloring of $[k]^k$ exists, we have shown $HJ(k,2) \leq k$.
\end{proof}

We now start on determining a good lower bound.

\noindent{\bf Notation.} For $k \in \mathbb{Z}^+$ and $t \in \mathbb{Z}^+ \cup \{0\}$, let $b(k,t)$ be the minimum integer $n$ such that
any $2$-coloring of $[1,n]$ admits either a $k$-ap of the first color or 2 integers $i$ and $j$,
both of the second color, with $|i-j| < k-t$.

\begin{lemma}\label{lem3} Let $k \in \mathbb{Z}^+$.  Let $t \in \mathbb{Z}^+ \cup \{0\}$ be the minimal
integer such that $k-t$ is prime. Then $b(k,t) > (k-t)(k-1)+1$.
\end{lemma}

\begin{proof} We will provide a valid 2-coloring of $[1,(k-t)(k-1)+1]$.
Consider the 2-coloring of $[1,(k-t)(k-1)+1]$ defined by coloring the congruence class $1$ modulo $k-t$
blue and all other integers red.  Since this coloring has no two   blue
integers with difference less than $k-t$ we need only show that this does not admit a red $k$-ap.  

Let
$a,a+d,\dots,a+(k-1)d$ be an arbitrary $k$-ap.  So that $a+(k-1)d \leq (k-t)(k-1)+1$ we must have $d < k-t$ or $d=k-t$ and $a=1$.
In the latter case, we see that $a$ is blue so that the $k$-ap is not red.  Hence, we may now assume $d<k-t$.
Since $k-t$ is prime, we have
$\gcd(d,k-t)=1$.  Hence, $xd \equiv 1-a \pmod{k-t}$ has a solution, say $x=j$,
with $0 \leq j \leq k-t-1$.  This means that $a+jd \equiv 1 \pmod{k-t}$ so that it has color blue.  Hence, for $d<k-t$ there
is no red $k$-ap.  
\end{proof}

\begin{rem} We can improve the lower bound to $(k-t)(k-1)+2t+1$, but that does not help below.
For $t=1$ this improved lower bound seems to be the exact value.
\end{rem}
This lemma allows us to state the following.

\begin{thm}\label{thm4} For $k \in \mathbb{Z}^+$, let $t \in \mathbb{Z}^+ \cup \{0\}$ be the minimal
integer such that $k-t$ is prime. Then $HJ(k,2) \geq k-t$.
\end{thm}

\begin{proof} From Lemma \ref{lem3}, let $\chi$ be a 2-coloring of $[1,(k-t)(k-1)+1]$ with no red $k$-ap
and no two blue integers with difference less than $k-t$.  We use this to color the points in $[k]^{k-t-1}$.
By showing this coloring admits no red $k$-line and no blue $2$-line, we will have $HJ(k,2) \geq k-t$, thereby
proving the theorem.

Assuming the alphabet is $\{0,1,\dots,k-1\}$, using the van der Waerden mapping we give the point
$(a_1,a_2,\dots,a_{k-t-1}) \in [k]^{k-t-1}$ the color $\chi\left(1+\sum_{i=1}^{k-t-1} a_i\right)$.

Let $w(x)$ be an arbitrary variable word with $d \in \{1,2,\dots,k-t-1\}$ occurrences of the variable $x$ and
denote by $w_i(j)$ the $i$th integer in $w(j)$.
Let $I$ be the set of indices where $x$ does not occur.  Let $a=1+\sum_{I} a_i$.  Then
$\sum_{i=1}^{k-t-1} w_i(j) = a+jd$ for all $j \in \{0,1,\dots,k-1\}$.  By definition of $\chi$, these cannot all be
red, else we have a red $k$-ap in $[1,(k-t)(k-1)+1]$ under $\chi$.

To see that we have no blue $2$-line, assume that $w(j)$ and $w(j+1)$ are both blue.
Then $\sum_{i=1}^{k-t-1} (w_i(j+1)-w_i(j)) = d$.  Since $d \leq k-t-1$, we have a contradiction since
$\chi$ does not admit two blue integers with difference less than $k-t$.
\end{proof}

\begin{cor} If $p \in \mathbb{Z}^+$ is prime, then $HJ(p,2)=p$.
\end{cor}

\begin{rem}\label{rem16} The van der Waerden mapping introduced in Definition~\ref{def:vdwmap} can be used in the off-diagonal situation and leads to the bound
$$
HJ(k,\ell) \geq \left\lfloor\frac{w(k,\ell)-2}{k-1}\right\rfloor+1, \quad k \geq \ell.
$$
This holds by taking a valid coloring of $[1,w(k,\ell)-1]$, using the van der Waerden mapping, and noting that
a monochromatic $\ell$-line maps to a monochromatic arithmetic progression.  Since
the coloring is valid, there is no monochromatic $k$-line of the first color and
no monochromatic $\ell$-line of the second color.
\end{rem}

 From Theorems \ref{thm1} and \ref{thm4} and a result of Ramanujan
\cite{Ram}, we have that $\frac{k}{2} \leq HJ(k,2) \leq k$ for all $k$ so that it is   linear, while, because
Green \cite{Gre} has shown that $w(k,3) > k^{c \left(\frac{\log k}{\log \log k}\right)^{1/3}}(1+o(1))$, we know that $HJ(k,3)$ grows faster
than any polynomial (see   Remark \ref{rem16}, above).

From \cite{BHP}, we know that in Theorem \ref{thm4} we have $t \leq k^{21/40} $ for sufficiently large $k$.
This allows us to state the following result.

\begin{thm} We have
$HJ(k,2) = k(1-o(1)).$
\end{thm}

  In Ramsey theory, lower bounds are typically easier; however, we have been unable to find a
lower bound matching our upper bound given in Theorem \ref{thm1}.  All computations ($k \in \{2,3,4,5,6,7\}$; see Table 3 in Section~\ref{sec:comp}) point to the following conjecture being true, although we have been unable to determine the values for $k \geq 8$.

\begin{conj} For all $k \in \mathbb{Z}^+$, we have $HJ(k,2)=k$.
\end{conj}

The difficult part of this conjecture is that when $k$ is composite, any underlying group structure
used (e.g., addition modulo $k$) necessarily has, by Cauchy's Theorem, an element with order
less than $k$ so that combinatorial lines transformed by the group's binary operation do not
necessarily hit every element of the group.  Hence, defining colorings based on a certain element
of the group will necessarily have lines that can avoid that element.  Furthermore, if we
select more than one element in the group, we have adjacent points on some combinatorial
lines that map to each element. 

If we view the hypercube as a poset with $x\preccurlyeq y$ meaning $y-x \in \{0,1\}^d$ and for those $i \in I$
where $y_i=x_i+1$, we have $x_i=x_j$ for all $i,j \in I$, we are searching
for an antichain $A$ such that every chain of length $k$ contains an element of $A$.
Since $(0,0,\dots,0), \dots, (k-1,k-1,\dots,k-1)$ is a maximal chain of length $k$, while
$(0,0,\dots,0), (1,0,0,\dots,0), \dots, (k-1,0,0,\dots,0), (k-1,1,0,\dots,0), \dots, (k-1,k-1,0,\dots,0),\dots, (k-1,k-1,\dots,k-1)$
is a maximal chain of length $dk$, this poset is not graded (since not all maximal chains have
the same length).  Hence, there is no consistent rank function that can be used  
to  define a valid coloring (since we would be, in this setting, searching
for a set of antichains that intersects certain chains of length $k$ and
elements with the same rank are the natural candidates).

\section{Combinatorial and Geometric Lines}

We next extend our approach to include geometric lines, which capture configurations beyond the classical setting.

\begin{defn}[CG-Word, CG-Line] For $k \in \mathbb{Z}^+$, let $v(x)$ represent a variable word
over the alphabet $\{0,1,\dots,k-1,x,k-1-x\}$ with at least one $x$ occurring (but $k-1-x$ need not occur).
We will call such a word a {\it cg-word}.
We will call the sequence $v(0), v(1),\dots,v(k-1)$ a {\it cg-line} (for combinatorial/geometric line).  A {\it partial cg-line} is defined
similarly to a partial combinatorial line.
Lines that include at least one instance of
$k-1-x$ are called {\it geometric lines}.

\end{defn}

\begin{defn}[CG-Hales-Jewett Number]\label{def20}
For positive integers $k_1,k_2,  \dots, k_r$
let $k = \max_i k_i$ and define the
{\it CG-Hales-Jewett number}
$HJ^*(k_1,k_2,\dots,k_r;r)$ to be the least integer $n$ such that every $r$-coloring of $[k]^n$ contains  
a monochromatic cg-line of color $i$ and length $k_i$ for some $i \in \{1,2,\dots,r\}$.
In the situation where all $k_i$'s are equal we use the notation 
$HJ^*(k;r)$.
When $r = 2$, we may write $HJ^*(k,\ell)$ in place of $HJ^*(k,\ell;2)$ for brevity.
\end{defn}

{If we divorce combinatorial and geometric lines,
is there a guarantee of a mono-chromatic geometric line without 
consideration for monochromatic combinatorial lines?
Yes.  We can  follow the proof of the Hales-Jewett Theorem
(see, e.g., \cite{Rob})
using geometric words instead of combinatorial words
and show that the base cases hold, where $G(k;1)=2$ and $G(2;r) \leq r+1$ are easy to show. 
We can also appeal to the Graham-Rothschild Theorem \cite{GR} by taking a 
diagonal of a guaranteed monochromatic hyperplane. This allows us to state the
following result.
}

\begin{thm}[Geometric Hales-Jewett]\label{thm21} Let $r \in \mathbb{Z}^+$ and $k_1,k_2,  \dots, k_r \in \mathbb{Z}^+$
with $k = \max_i k_i$.
There exists a minimal number 
$n=G(k_1,k_2,\dots,k_r;r)$ such that every $r$-coloring of $[k]^n$ contains  
a monochromatic geometric line of color $i$ and length $k_i$ for some $i \in \{1,2,\dots,r\}$. 
\end{thm}

As previously done, in the situation when $k=k_1=k_2=\cdots=k_r$
we use the notation $G(k;r)$; furthermore,
when $r = 2$, we  write $G(k,\ell)$ in place of $G(k,\ell;2)$ for brevity.

Returning to cg-lines, it is known that Tic-Tac-Toe with 2 players on a
$3 \times 3 \times 3$ board always has a winner, regardless
of player strategy (see, e.g., Exercise 3.3 in \cite{Beck}).
This is also given by our computation that verifies this: $HJ^*(3;2)=3$; see Table \ref{tab:cghj} in Section 6.

We note that $HJ^*(3;r)$ is the true high-dimensional $r$-player Tic-Tac-Toe analogue. This is because certain winning lines (e.g. $\{(1,3), (2,2), (3,1)\}$ on $[3]^2$) are geometric lines but not combinatorial lines. This explains why $HJ^*(3;2) = 3$ allows us to state that Tic-Tac-Toe cannot end in a tie on a $3 \times 3 \times 3$
board even though $HJ(3;2)=4$..

Computational results on $HJ^*(k;2)$ for small values
of $k$ can be found in Section \ref{sec:comp}.
We can also offer the following formulas for
$(k,r) \in \{(k,1), (2,r)\}$.

\begin{lemma}\label{lem13}  We have
$HJ^*(k;1)=1$  
and
$HJ^*(2;r) = \lceil \log_2 (r+1)\rceil$.
\end{lemma}

\begin{proof}
The fact that $HJ^*(k;1)=1$ is trivial.
For $HJ^*(2;r)$, first
note that every pair of points are on a unique cg-line
by the following.  In $\{0,1\}^d$, let $p$ and $q$ be points.  We will build a word $w(x)$ based on them.
If $p_i=q_i$ let $w_i=p_i$; if $p_i<q_i$ then let $w_i=x$; if $p_i>q_i$ then let $w_i=1-x$.
Note that $w(0)=p$ and $w(1)=q$.  

Hence, if any two points have the same color, we have a monochromatic cg-line.
This means that every color can only be used once if we want to avoid
a monochromatic line.  Since there are $2^d$ points in $\{0,1\}^d$,
we must have at least $2^d$ colors to avoid a monochromatic geometric line.
Hence, $HJ^*(2;2^d) > d$ while $HJ^*(2;2^d-1) \leq d$. Thus,
$HJ^*(2;r) > \lfloor \log_2 r \rfloor $ and $HJ^*(2;r) \leq \lceil \log_2 (r+1)\rceil$.  Hence,
$$
\lfloor \log_2 r \rfloor < HJ^*(2;r) \leq \lceil \log_2 (r+1)\rceil,
$$
Now let $k=\lfloor \log_2 r \rfloor $ so that $2^k \leq r < 2^{k+1}$.  Then $2^k<r+1 \leq 2^{k+1}$.
Thus $k+1=1+\lfloor \log_2 r \rfloor =\lceil \log_2 (r+1)\rceil$.
Hence, $HJ^*(2;r)=\lceil \log_2 (r+1)\rceil$.
\end{proof}

In \cite{BPV} the following
bound on the CG-Hales-Jewett number is given: 
$$HJ^*(k;2) > \frac{2^{\frac{k}{4}}}{3k^4}(1+o(1)),$$ 
as well
as the statement that $HJ^*(4;2)\geq 5 $ is the best known lower bound (as of 2009).
We have discovered no other published lower bound on $HJ^*(4;2)$. Using computational methods as described in Section~\ref{sec:comp}, we have an improved lower bound of $HJ^*(4;2) \geq 7$, along with a newly discovered lower bound  $HJ^*(3;3)\geq 7$, meaning 3-player Tic-Tac-Toe requires at least 7 dimensions
to guarantee a winner.

We can improve the asymptotic lower bound 
on $HJ^*(k;2)$ given in \cite{BPV}.  We cannot use the  van der Waerden mapping   since certain geometric lines
(those with an equal number of forward and backward variables) map to a constant and hence
are monochromatic under the van der Waerden mapping.

\begin{thm} For $k \geq 2$, we have
$$
HJ^*(k;2) > \frac{2^{\frac{k-1}{3}}-1}{(k-1)^{\frac23}}.
$$
\end{thm}

\begin{proof} Since we have $HJ^*(2;2) = 2, HJ^*(3;2) = 3$, and $HJ^*(4;2) \geq 7 $, we may assume $k \geq 5$.
With a goal of using Lemma \ref{genlem}, as done in \cite{BPV} we map $(a_1,a_2,\dots,a_d) \in [0,k-1]^d$ to $1+\sum_{i=1}^d a_i^2$.
As a result of this mapping, the combinatorial and geometric lines are mapped to
$$
a+ix^2+j(k-1-x)^2
$$
for some $i \geq 1$ and $j \geq 0$, with $i+j \leq d$.

We now consider 2-colorings of $[1,n]$ with no monochromatic sequence
of the above form.  Letting $M(k)$ be the largest such integer $n$,
via computer (program {\tt HJ.f}\footnote{Available at the third author's website {\tt https://math.colgate.edu/$\sim$aaron}}), we have
$$
\begin{array}{l|l}
k&M(k)\\\hline
2&2 \\
3&15 \\
4& 56\\
5& \geq 133 \\
\end{array}
$$

We first show that $M(k) \geq 4(k-1)^2-1$ for $k \geq 5$ as this will be used below.  To see this, consider
the 2-coloring that colors $[1,(k-1)^2-1] \cup [2(k-1)^2, 3(k-1)^2-1]$ red and its
complement in $[1,4(k-1)^2-1]$ blue.

Consider $a+ix^2+j(k-1-x)^2$ for $x=0,1,\dots,k-1$.  For $i \geq j$, the largest term of
the progression is $a+i(k-1)^2$.  
If $a \in [1,(k-1)^2]$ then, since $i \geq 1$, we have
that the largest term of the progression must be in $[2(k-1)^2, 3(k-1)^2-1]$ so that
$i=2$.
The largest possible difference between consecutive terms of the sequence
is $(2k-3)i - j \leq (2k-3)i = 4k-6$.  Since $(k-1)^2 > 4k-6$ for $k\geq 5$, some term of the
progression must lie in $[(k-1)^2, 2(k-1)^2-1]$ so that the progression is not red.
If $a \in [2(k-1)^2, 3(k-1)^2-1]$, then $a+(k-1)^2\geq 3(k-1)^2$ so that the progression
is not red.

For $i<j$, the analysis is very similar and left to the reader, as is the fact that there
is no blue progression.

We now turn to the basic probabilistic method to produce a lower bound.
Since the probability of a given progression being monochromatic when the integers are
randomly colored is $2^{1-k}$, if we let $n(k)$ be the number of progressions in $[1,M]$
and determine $M=M(k)$ such that $n(k)2^{1-k} < 1$, then we have the existence of a 2-coloring
with no such monochromatic progressions.

We next determine $n(k)$.  We have
$$
n(k) = \sum_{i=1}^d \sum_{j=0}^{d-i} \sum_{a=1}^{\min(M-i(k-1)^2,\, M-j(k-1)^2} 1.
$$
Hence, we evaluate
$$
\sum_{i=1}^{d/2} \sum_{j=0}^{i} \sum_{a=1}^{M-i(k-1)^2} 1 + \sum_{i=d/2}^{d} \sum_{j=0}^{d-i} \sum_{a=1}^{M-i(k-1)^2} 1+ \sum_{i=1}^{d/2} \sum_{j=i}^{d-i} \sum_{a=1}^{M-j(k-1)^2} 1.
$$
A tedious, but routine calculation gives us 
$$
n(k) \leq \frac{M(d+1)^2}{2} - (k-1)^2 \frac{d^3}{12}.
$$
Now, we know that $1 \leq d \leq \frac{M-1}{(k-1)^2}$ so that the progression lies inside $[1,M]$.  Using these, we have 
$$
n(k)\leq \frac{M^3}{2(k-1)^4}+\frac{M^2}{(k-1)^2} + \frac{M}{2} - (k-1)^2 \frac{d^3}{12} \leq \frac{M^3}{(k-1)^4} -  \frac{(k-1)^2}{12},
$$
where the last inequality holds
provided $M\geq (1+\sqrt{2})((k-1)^2+1)$, which holds since $(1+\sqrt{2})(k-1)^2 < 4(k-1)^2-1
\leq M$
for all $k \geq 3$.

Hence, if 
$$
 \frac{\frac{M^3}{(k-1)^4} -  \frac{(k-1)^2}{12}}{2^{k-1} }< 1
$$
then we have $n(k)2^{1-k} < 1$ meaning there exists a 2-coloring of $[1,M]$ with no monochromatic progression.  This occurs for
$$
M = 2^{\frac{k-1}{3}}(k-1)^{\frac43}.
$$

We now use the mapping for $HJ^*(k;2)$ and note that the largest value under the mapping is $1+d \cdot (k-1)^2$.
Hence, if
$M \geq 1+ d \cdot (k-1)^2$ then we obtain a valid coloring of $[k]^d$.  This gives
$$
HJ^*(k;2) > \frac{M-1}{(k-1)^2} = \frac{2^{\frac{k-1}{3}}}{(k-1)^{\frac23}}-\frac{1}{(k-1)^2}>
\frac{2^{\frac{k-1}{3}}}{(k-1)^{\frac23}}-\frac{1}{(k-1)^{\frac23}} =\frac{2^{\frac{k-1}{3}}-1}{(k-1)^{\frac23}}.
$$
\end{proof}

\subsection{Connect Four}

As we have computed $HJ^*(4;2) \geq 7$, in order to
have a 4-in-a-row winner in a 2-player game on a hypercube of side length $4$
we require at least 7 dimensions.  But what if we
allow the side length to be larger than 4?

The classic game Connect Four by Milton Bradley is played with red and black
checkers on a $7 \times 6$ board.  If you have any playing
experience with this game, you have probably ended in a
tie before, even though it is known \cite{Allis} that the
first player has a winning strategy.  Here we investigate when
we are guaranteed a winner, regardless of playing strategy.
For simplicity, we assume a board with equal side lengths.

\begin{defn}[Connect Four Number]\label{def24} For $k,r \in \mathbb{Z}^+$, let the {\it Connect Four number} $CF(k;r)$ be
the minimal dimension $d$  such that for any $n \geq d$, every $r$-coloring
of the points in $[k]^n$ admits a monochromatic partial cg-line
of length 4.
\end{defn}

It is easy to see that $CF(k;2)>2$ for all $k$ by the following
2-coloring \small
$$\begin{array}{cccccc}
0&1&0&1&0&\cdots\\
0&1&0&1&0&\cdots\\
0&1&0&1&0&\cdots\\
1&0&1&0&1&\cdots\\
1&0&1&0&1&\cdots\\
1&0&1&0&1&\cdots\\
0&1&0&1&0&\cdots\\
0&1&0&1&0&\cdots\\
0&1&0&1&0&\cdots\\
\vdots&\vdots&\vdots&\vdots&\vdots&\ddots
\end{array}
$$\normalsize
so that the 2-dimensional game may end in a tie regardless of
the board size.  

With the help of the Maple program {\tt ConnectFour}\footnote{Available at the third author's website {\tt https://math.colgate.edu/$\sim$aaron}} we also have $CF(k;2)>3$ for all $k$
via the following coloring.  For $(a,b,c) \in [k]^3$, define
$f(a,b,c)=a+3b+5c \pmod{11}$.  If $f(a,b,c) \in \{0,1,5,7,8,10\}$ we
color $(a,b,c)$ red; otherwise, we color it blue.  The Maple
program constructs all combinatorial and geometric 4-lines 
and shows all such 4-lines are bichromatic.  Hence, a 3-dimensional
version of Connect Four can end in a tie regardless of board size, in contrast to
3-dimensional Tic-Tac-Toe which must always have a winner on a
$3 \times 3 \times 3$ board.

For higher dimensions, we turn to SAT solvers (see Section \ref{sec:comp}) to
show the following result, where we find that Connect Four 
can end in a tie in 5 dimensions on any sized board.

As explained in the next proof, we provide a method
for coloring the entire 5-dimensional space based
on a finite hypercube.

{
\begin{thm}  
We   have $CF(k;2) \geq 6$ for all $k$.
\end{thm}

\begin{proof}  We start by explaining why $CF(k;2) \geq 6$ for all $k$. Let $w(x)$ be a 5-dimensional cg-word.  Via the
SAT solver Kissat, we discover a 2-coloring of
$[15]^5$ for which none of
$$
(w(i \!\!\!\mod{15}), w(i+1 \!\!\!\mod{15}), w(i+2 \!\!\!\mod{15}), w(i+3 \!\!\!\mod{15})), \quad 1 \leq i \leq 15,
$$
are monochromatic.  Hence, we can color all points in
$(\mathbb{Z}^+)^5$ with the discovered coloring by
giving $(x_1,x_2,\dots,x_5)$ the color
$(x_1 \!\!\!\pmod{15}, x_2\!\!\!\pmod{15}, \dots, x_5 \!\!\!\pmod{15})$ and be guaranteed that no monochromatic
cg-line of length 4 exists.  Thus, given
$k \in \mathbb{Z}^+$, by considering this
coloring restricted to $[k]^5$, we have
$CF(k;2) \geq 6$ for any $k$.
\end{proof}

\begin{rem} Our computer run for $k=6$ in 6 dimensions did not
finish before our allotted time ran out.  It ran for 600 hours with
no solution.
\end{rem}

}

\section{Computational Results}\label{sec:comp}

Given a Boolean formula $\varphi: \{0,1\}^n \rightarrow \{0,1\}$, the Boolean Satisfiability (SAT) problem seeks to determine whether there exists an assignment to its variables under which the formula evaluates to true.

While SAT is NP-complete, SAT solvers use Conflict-Driven Clause Learning (CDCL) to often solve instances far faster than what worst-case complexity bounds would suggest. Consequently, problems that can be reduced to SAT may be approached efficiently using these methods.

The problem of determining whether there exists an $r$-coloring of $[k]^n$ that avoids monochromatic combinatorial lines can be reduced to SAT. The specific encoding depends on the parameters $k$ and $r$; we outline the reduction for the cases $r = 2$ and $r = 3$.

Let $\mathcal{L}$ denote the set of combinatorial lines in $[k]^n$, and let $L = \{w(1), w(2), \dots, \break w(k)\} \in \mathcal{L}$ denote a combinatorial line.

\noindent{\bf Notation.}
For $L \subseteq [k]^n$ and color $c \in [r]$ define the Boolean clauses
\[
\mathbf{b}_{L,c} := \bigvee_{i \in [k]} b_{w(i),c},
\qquad
\neg \mathbf{b}_{L,c} := \bigvee_{i \in [k]} \neg b_{w(i),c},
\]  
where $b_{w(i),c}=1$ if $w(i)$ has color $c$ and equals 0 otherwise.

These clauses, respectively, encode that at least one point in $L$ receives color $c$ and that not all points in $L$ receive color $c$.

\subsection{Encoding Two Colors}
When $r = 2$, each point $x \in [k]^n$ is assigned a Boolean variable $b_{x,1}$. The Boolean formula encoding $HJ(k;2) > n$ is
\[
\bigwedge_{L \in \mathcal{L}}
\left[
\neg \mathbf{b}_{L,1}
\;\wedge\;
\mathbf{b}_{L,1}
\right].
\]

The clause $\neg \mathbf{b}_{L,1}$ forbids monochromatic lines of the first color, while $\mathbf{b}_{L,1}$ forbids monochromatic lines of second color.

\subsection{Encoding More than Two Colors}
When $r = 3$, each point $x \in [3]^n$ is assigned a Boolean triple $(b_{x,1}, b_{x,2}, b_{x,3})$. The Boolean formula is
\[
\bigwedge_{L \in \mathcal{L}} \Bigg[
\bigwedge_{c \in [3]} \neg \mathbf{b}_{L,c}
\;\wedge\;
\bigwedge_{i \in [k]} 
\left( \bigvee_{c \in [r]} \mathbf{b}_{w(i),c}\right)
\Bigg].
\]

The clause $\neg \mathbf{b}_{L,c}$ forbids a monochromatic line of color $c$, whereas the clauses $\displaystyle \bigvee_{c \in [r]} \mathbf{b}_{w(i),c}$ guarantee point $w(i)$ is assigned a color.

Encodings for $r \geq 4$ follow by straightforward extensions of these constructions. Similarly, Hales-Jewett variants can be encoded via natural modifications of the $HJ(k;r)$ formulation. For further details on the implementation, see our GitHub Repository \cite{Repo}. 

\subsection{Some Values and Bounds}

With the computational approach outlines above in this section, we obtain new exact values and bounds for several Hales-Jewett variants. Our computations were carried out using the SAT solvers Kissat \cite{Kissat} and  CryptoMiniSat \cite{CryptoMiniSat}. The results are summarized in Tables~\ref{tab:cghj} and \ref{tab:offdiaghj}.

\begin{table}[ht!]
\centering
\renewcommand{\arraystretch}{1.0}
\small
\begin{tabular}{>{\centering\arraybackslash}m{1.0cm}|
                >{\centering\arraybackslash}m{1.5cm}
                >{\centering\arraybackslash}m{1.7cm}
                >{\centering\arraybackslash}m{1.7cm}
                }
\toprule
{\boldmath${k}$}  & $G(k;2)$ & $HJ^*(k;2)$&$HJ^*(k;3)$ \\
\midrule
\textbf{2}   & 3 & 2& 2  \\
\textbf{3}   & 5 & 3 & $\geq 7$ \\
\textbf{4}  &  $\geq 8 $ & $\geq 7$ & $\geq 10$ \\
\bottomrule
\end{tabular}
\caption{Values and bounds for geometric and CG-Hales-Jewett numbers}
\label{tab:cghj}
\end{table}

\begin{table}[ht!]
\centering
\renewcommand{\arraystretch}{1.2}
\small
\begin{tabular}{>{\centering\arraybackslash}m{1.2cm}|
                >{\centering\arraybackslash}m{1.5cm}
                >{\centering\arraybackslash}m{1.5cm}
                >{\centering\arraybackslash}m{1.5cm}
                >{\centering\arraybackslash}m{1.5cm}
                >{\centering\arraybackslash}m{2.0cm}
                >{\centering\arraybackslash}m{2.0cm}}
\toprule
{\boldmath${k}$}  & $HJ(k,2)$ & $G(k,2)$ & $HJ^*(k,2)$\\
\midrule
\textbf{2} & 2 & 3 & 2 \\
\textbf{3}  & 3 & 4 & 2 \\
\textbf{4} & 4 &  4& 3 \\
\textbf{5} & 5 & 6  & 4 \\
\textbf{6} & 6 &  6 & 4  \\
\textbf{7}  & 7& $\geq 8$ & 5 \\
\textbf{8}  & $\geq 7$  & 7 & 4 \\
\textbf{9}   & $\geq 7$ &  $\geq 9$ & $\geq 6$ \\
\textbf{10}  & $\geq 7$ & $7$ & 5  \\
\bottomrule
\end{tabular}
\caption{Values and bounds for 2-color off-diagonal Hales-Jewett variants}
\label{tab:offdiaghj}
\end{table}

Additionally, we have the following values and bounds for the 3-color and 4-color off-diagonal Hales-Jewett numbers. We have shown:
\[HJ(3,2,2;3) = 5 \quad \text{ and }\quad  HJ(3,3,2;3) = 9;\] and
\[HJ(3,2,2,2;4) = 7 \quad \text{ and } \quad HJ(3,3,2,2;4) \geq 11.\]

\begin{rem}
Hales-Jewett SAT instances exhibit unusual branching behavior. Disabling variable activity updates via the Kissat \texttt{--no-bump} flag leads to a substantial reduction in conflicts, especially for  $HJ(3;3)$ instances. Tables \ref{tab:hj33vsids} and \ref{tab:hj42vsids} outline these results.
(The values in the No Bump column of Table \ref{tab:hj33vsids} are not typos.)
Note that, in general, fewer conflicts leads to quicker completion.

\begin{table}[H]
\centering

\begin{minipage}[b]{0.48\textwidth}
\centering
\begin{tabular}{c|cc}
\toprule
 & \multicolumn{2}{c}{Conflicts} \\
$n$ & Baseline & No Bump \\
\midrule
7  & 76     & 97 \\
8  & 810K   & 258 \\
9  & 3.55M  & 693 \\
10 & 568M   & 1.31M \\
\bottomrule
\end{tabular}\caption{Kissat Conflicts on $HJ(3;3)$}\label{tab:hj33vsids}
\end{minipage}
\hfill
\begin{minipage}[b]{0.48\textwidth}
\centering
\begin{tabular}{c|cc}
\toprule
 & \multicolumn{2}{c}{Conflicts} \\
$n$ & Baseline & No Bump \\
\midrule
7  & 226K  & 12.8K \\
8  & 721K  & 249K \\
9  & 7.61M & 1.97M \\
\bottomrule
\end{tabular}\caption{Kissat Conflicts on $HJ(4;2)$}\label{tab:hj42vsids}
\end{minipage}

\end{table}

Further investigation of this phenomenon may lead to more efficient SAT solutions for Hales-Jewett instances.

Our results demonstrate that computational methods, particularly SAT-based approaches, provide a viable pathway for advancing the study of Hales-Jewett numbers. We are optimistic that the interplay between combinatorial structure and computation will continue to yield new insights into these computationally difficult problems.

\end{rem}

\section{Open Questions}

Several of the bounds obtained in this work may not be tight, leaving room
for improvement. We conclude by highlighting  open problems raised by this work.

\noindent{\bf Conjecture A.}
For all $k \in \mathbb{Z}^+$, we have $HJ(k,2)=k$.

\noindent{\bf Conjecture B.}
For all $k \in \mathbb{Z}^+$, we have $HJ^*(k,2) < k$.
 
{
 \noindent{\bf Question C.} What is the minimum dimension $d$ such that
 there exists $k \in \mathbb{Z}^+$ such that
the 2-player Connect Four game always has a winner on the $[k]^d$ board? 
 }

 {
 \noindent{\bf Question D.}
Does the following hold for all $k$: $HJ(k,2) \leq G(k,2)$?
Note that if this is true, then Conjecture A is false since
$G(8,2)=G(10,2)=7$.
 }
 
 {
 \noindent{\bf Question E.}
 What can be said about $2$-colorings of discrete hypercubes that either contain
 a combinatorial line of length $k$ of the first color or a geometric line of
 length $\ell$ of the second color?  In particular, do these numbers always exist?
 }

\vskip 20pt
\noindent
\textbf{Acknowledgments.} This work was supported in part by the Colgate University Research Computing and Data services and the Colgate Supercomputer (partially supported by NSF grant OAC-2346664).

\end{document}